\documentclass[
final
 , nomarks
]{dmtcs-episciences}

\usepackage[utf8]{inputenc}
\usepackage{subfigure}
\usepackage{amsmath,amssymb}

\usepackage[round]{natbib}

\author{Mirjana Mikala\v{c}ki
  \and Milo\v{s} Stojakovi\'{c}}
\title{Fast strategies in biased Maker--Breaker games \thanks{Research partly supported by Ministry of Education, Science and Technological Development, Republic of Serbia (Grant No.\ $174019$) and Provincial Secretariat for Higher Education and Scientific Research, Province of Vojvodina (Grant No.\ $142$-$451$-$2787/2017$).}}
\affiliation{
University of Novi Sad, Faculty of Sciences, Department of Mathematics and Informatics, Serbia.}
\received{2017-10-31}
\revised{2018-7-17}
\accepted{2018-9-21}

\DeclareFontFamily{OT1}{pzc}{}
\DeclareFontShape{OT1}{pzc}{m}{it}{<-> s * [1.10] pzcmi7t}{}
\DeclareMathAlphabet{\mathpzc}{OT1}{pzc}{m}{it}

\newtheorem{theorem}{Theorem}[section]
\newtheorem{lemma}[theorem]{Lemma}
\newtheorem{claim}[theorem]{Claim}

\newtheorem{corollary}[theorem]{Corollary}

\numberwithin{equation}{section}

\newcommand\cref[1]{Corollary~\ref{cor:#1}}

\newcommand\cG{{\mathcal G}}

\newcommand\cP{{\mathcal P}}

\newcommand\cE{{\mathcal E}}

\newcommand{\danger}{\mathsf {dang}}
\newcommand{\avdanan}{\overline{\danger}}

\begin{document}
\publicationdetails{20}{2018}{2}{6}{4033}
\maketitle
\begin{abstract}
  We study the biased $(1:b)$ Maker--Breaker positional games, played on the edge set of the complete graph on $n$ vertices, $K_n$. Given Breaker's bias $b$, possibly depending on $n$, we determine the bounds for the minimal number of moves, depending on $b$, in which Maker can win in each of the two standard graph games, the Perfect Matching game and the  Hamilton Cycle game.
\end{abstract}

\section{Introduction}

In a Maker--Breaker positional game, a finite set $X$ and a family $\cE$ of subsets of $X$ are given, and two players, Maker and Breaker, alternate in claiming unclaimed elements of $X$ until all the elements are claimed, with Breaker going first. Maker wins if she claims all elements of a set from $\cE$, and Breaker wins otherwise. The set $X$ is referred to as the \emph{board}, and the elements of $\cE$ as the \emph{winning sets}. As Maker--Breaker positional games are finite games of perfect information and no chance moves, we know that in every game one of the players has a winning strategy. More on various aspects of positional game theory can be found in the monograph of~\cite{BeckBook} and in the recent monograph of~\cite{HKSSbook}.

We are interested in positional games on graphs, where the board $X$ is the edge set of a graph, and we will mostly deal with games played on the edge set of the complete graph $E(K_n)$. Probably three most standard positional games are \emph{the Connectivity} game, where Maker wants to claim a spanning tree, \emph{the Perfect Matching} game, where the winning sets are all perfect matchings of the base graph, and \emph{the Hamilton Cycle} game, where Maker's goal is to claim a Hamilton cycle.

Once the order $n$ of the base graph gets large, it turns out that Maker can win in each of the three mentioned games in a straightforward fashion. But our curiosity does not end there, as there are several standard approaches to make the setting more interesting to study. One of them is the so-called \emph{biased games}, where Breaker is given more power by being allowed to claim more than one edge per move. The other approach we focus on is the \emph{fast win} of Maker, where the question we want to answer is not just if Maker can win, but also how fast can she win.

Given a positive integer $b$, in the $(1:b)$ biased game Breaker claims $b$ edges in each move, while Maker claims a single edge. The parameter $b$ is called the bias. Due to the ``bias monotonicity" of Maker--Breaker games, it is straightforward to conclude that for any positional game there is some value $b_0 = b_0(n)$ such that Maker wins the game for all $b < b_0(n)$, while Breaker wins for $b \geq b_0(n)$ (see~\cite{HKSSbook} for details). We call $b_0(n)$ the \emph{threshold bias} for that game.

The biased games were first introduced and studied by~\cite{CE}, and some thirty years later the papers of~\cite{GS} and~\cite{K} finally located the leading term of the thresholds for the games of Connectivity, Perfect Matching and Hamilton Cycle, which turned out to be $n/\ln n$ for all three games.

Moving on to the concept of fast winning, when we know that Maker can win an unbiased game, a natural question that we can ask is -- what is the minimum number of moves for Maker to win? Questions of this type appeared frequently, often as subproblems, in classical papers on positional games, and the concept of fast Maker's win was further formalized by~\cite{HKSS2}. It is not hard to see that Maker can win the unbiased Connectivity game as fast as the size of the winning set allows, in $n-1$ moves. For the other two games it takes her a bit longer (one move longer, to be more precise) -- she can win the unbiased Perfect Matching game in $n/2 +1$ moves (for $n$ even) as shown by~\cite{HKSS2}, and the unbiased Hamilton Cycle game in $n +1$ moves, shown by~\cite{HS09}, and in both cases that is the best she can do. We note that some research has also been done on fast Maker's win in the unbiased $k$-Connectivity, Perfect Matching and Hamilton Cycle games played on the edge set of a random graph, see~\cite{CFKL}.

Knowing how fast Maker can win, and how to win fast, is important, as this often helps us when looking at other positional games. Indeed, there are numerous examples where a player's winning strategy may call for building a certain structure \emph{quickly} before proceeding to another task. Also, one of very few tools that proved to be useful when tackling the so-called strong positional games are the fast Maker's winning strategies, see~\cite{CM,FH1,FH2}.

Our goal in this paper is to combine the two presented concepts -- the biased games and the fast winning, looking into the possibilities for Maker to \emph{win fast in biased games}. In other words, given a game $\cG$ and a bias $b$ such that Maker can win the $(1:b)$ biased game, we want to know in how many moves he can win the game. One obvious lower bound for the duration of the game is the size of the smallest winning set, and that is $n-1$ for the Connectivity game, $\frac n2$ for the Perfect Matching game, and $n$ for the Hamilton Cycle game.

It is not hard to see that in the Connectivity game Maker does not ever need to close a cycle, and therefore, even in the biased game, whenever she can win she can do so in exactly $n-1$ moves. As for the  Perfect Matching game, the following theorem gives fast Maker's win for most of the range of biases for which Maker can win, up to the (order of the) threshold for Maker's win in the game.

\begin{theorem} \label{thm:perfM}
There exist $\alpha>0$ and $C>0$ such that for every $2\leq b\leq \frac{\alpha n}{\ln n}$, Maker
can win the $(1:b)$ Perfect Matching game played on
$E(K_n)$ within $\frac n2+Cb\ln b$ moves, for large
enough $n$.
\end{theorem}

Moving on to the Hamilton Cycle game, we can prove the following two results for fast Maker's win. The first one is more powerful, but it applies only for small values of bias, while the second one covers a wider range of bias.

\begin{theorem}\label{thm:HamVerySmall}
There exists $C>0$ such that for $b\geq 2$ and $b=o\left(\frac{\ln n}{\ln\ln n}\right)$, Maker can win the $(1:b)$ Hamilton Cycle game played on $E(K_n)$ within $n+Cb^2\ln b$ moves, for large
enough $n$.
\end{theorem}

\begin{theorem}\label{thm:HamSmall}
There exist $\delta > 0$ and $C>0$, such that for $b\leq\delta \sqrt{\frac{n}{\ln^5 n}}$ Maker can win the $(1:b)$ Hamilton Cycle game played on $E(K_n)$ within $n+ C b^2\ln^5 n$ moves, for large
enough $n$.
\end{theorem}

Finally, when the bias is large, we can apply the following result of~\cite{K}, as it provides Maker with a win within $14n$ moves in the Hamilton Cycle game, and thus also in the Perfect Matching game.

\begin{theorem} [\cite{K}, Theorem 1]\label{thm:Krivi}
Maker can win the $(1:b)$ Hamilton Cycle game played on
$E(K_n)$ in at most $14n$ moves, for every $b\leq
\left(1-\frac{30}{\ln^{1/4}n}\right)\frac{n}{\ln n}$, for large
enough $n$.
\end{theorem}

On the other hand, looking at the prospects of Breaker to postpone Maker's win, we can move away from the obvious lower bound in both Perfect Matching game and Hamilton Cycle game.

\begin{theorem}
\label{thm:delay} In $(1:b)$ Maker--Breaker game, for every bias $b$ and $n$ large enough, Breaker can postpone Maker's win
\begin{enumerate}[(i)]
\item in the Perfect Matching game for at least $\frac n2 + \frac b4$ moves,
\item in the Hamilton Cycle game for at least $n + \frac b2$ moves.
\end{enumerate}
\end{theorem}

To sum up, if the number of moves Maker needs to play in order to win in the Perfect Matching game is denoted by $p(b)$, on the whole range of biases between $1$ and $(1-o(1))n/\ln n$ we have that $\frac b4 \leq p(b) - \frac n2 \leq O(b \ln b) $, as given by Theorem~\ref{thm:perfM}, Theorem~\ref{thm:Krivi} and Theorem~\ref{thm:delay}~(i).

In the Hamilton Cycle game, if we denote the number of moves Maker needs to play in order to win by $h(b)$, then Theorem~\ref{thm:HamVerySmall}, Theorem~\ref{thm:HamSmall}, Theorem~\ref{thm:Krivi} and Theorem~\ref{thm:delay}~(ii) provide non-trivial upper and lower bounds for the whole range of biases between $1$ and $(1-o(1))n/\ln n$. If we look at the value $h(b) - n$ and express both the upper and lower bounds as functions of $b$, the lower bound on the whole range is $\frac b2$, while the upper bound varies between $b^{1+\varepsilon}$ and $b^{7+\varepsilon}$, for any $\varepsilon>0$. In particular, for $b$ a constant, both upper and lower bounds are a constant.

Finding the right order of magnitude of both $p(b) - \frac n2$ and $h(b) - n$ remains an open problem, and we are particularly curious if they are linear in $b$.

The rest of the paper is organized as follows. After we list the notation we use, in Section~\ref{sec::Aux} we collect some preliminaries. Then, in Section~\ref{sec:makerPM} we prove Theorem~\ref{thm:perfM}, in Section~\ref{sec:hamVS} we prove Theorem~\ref{thm:HamVerySmall}, in Section~\ref{sec:hamS} we prove Theorem~\ref{thm:HamSmall} and in Section~\ref{sec:breakfast} we prove Theorem~\ref{thm:delay}.

\subsection{Notation} Our graph-theoretic notation is standard and
follows that of~\cite{West}. In particular, we use the following.

For a graph $G$, let $V(G)$ and $E(G)$ denote its sets of vertices
and edges respectively, and let $v(G) = |V(G)|$ and $e(G) = |E(G)|$.
For a set $S \subseteq V(G)$,
let $G[S]$ denote the subgraph of $G$ which is induced on the set
$S$. For disjoint sets $S,T \subseteq V(G)$, let $N_G(S,T) = \{u \in T : \exists v \in S, \{u,v\} \in E(G)\}$ denote the set of neighbors of the vertices of $S$ in $T$. For a set $T \subseteq V(G)$ and a vertex $w \in V(G) \setminus T$ we abbreviate $N_G(\{w\}, T)$ to
$N_G(w, T)$, and let $d_G(w,T) = |N_G(w,T)|$ denote the degree of
$w$ into $T$. For a set $S \subseteq V(G)$ and a vertex $w \in V(G)$
we abbreviate $N_G(S, V(G) \setminus S)$ to $N_G(S)$ and $N_G(w,
V(G) \setminus \{w\})$ to $N_G(w)$. 
We let $d_G(w) = |N_G(w)|$ denote the degree of $w$ in $G$. The minimum and maximum degrees of a graph $G$ are denoted by $\delta(G)$ and
$\Delta(G)$ respectively. Often, when there is no risk of confusion,
we omit the subscript $G$ from the notation above. Given a path $P$, let $v_P^1$ and $v_P^2$ denote its endpoints (in arbitrary order).
A Hamilton path in a given graph $G$ is a path containing all the vertices of $G$. A graph $G$ is called \emph{Hamilton-connected} if for every $p,q \in V(G)$ there exists a Hamilton path in $G$ between $p$ and $q$.

Assume that some Maker--Breaker game, played on the edge set of some
graph $G$, is in progress. At any given moment during this game, we
denote the graph spanned by Maker's edges by $M$ and the graph
spanned by Breaker's edges by $B$; the edges of $G \setminus (M \cup
B)$ are called \emph{free}.

In the rest of our paper $\ln$ stands for the natural logarithm. For the clarity of presentation we omit floor and ceiling signs whenever they are not crucial.

\section{Preliminaries}
\label{sec::Aux}

We need the results about the so-called \emph{Box game}, introduced by~\cite{CE}. The game is played on $k$ disjoint winning sets, whose sizes differ by at most $1$, that contain altogether $t$ elements. BoxMaker claims $a$ elements per move, while BoxBreaker claims $1$ element per move. The game is denoted by $B(k,t,a,1)$. In order to give a criterion for winning in $B(k,t,a,1)$, the following recursive function was defined in~\cite{CE}
\begin{displaymath}
f(k,a):=\left\{
\begin{array}{rl}
0, & k=1\\
\lfloor\frac{k(f(k-1,a)+a)}{k-1}\rfloor, & k\ge 2.
\end{array}\right.
\end{displaymath}

The value of $f(k,a)$ can be approximated as
\begin{equation}
\label{eq:approx}
(a-1)k\sum_{i=1}^k\frac{1}{i}\le f(k,a) \le ak\sum_{i=1}^k\frac{1}{i}.
\end{equation}

\begin{theorem}[\cite{CE}, \emph{the Box game criterion}]
\label{thm:box}
Let $a$, $k$ and $t$ be positive integers. BoxMaker has a winning strategy in $B(k,t,a,1)$ if and only if $t\le f(k,a)$.
\end{theorem}

The following theorem deals with the $(1:b)$ Maker--Breaker game
played on the edge set of a graph $G$. Roughly speaking, it claims that if the minimum degree of $G$ is not too small, Maker can build a spanning subgraph with large
minimum degree fast, while making sure that throughout the game, as
long as a vertex $v$ is not of large degree in Maker's graph, the
proportion between the number of Maker's and Breaker's edges touching $v$ is
``good". The proof is similar (in fact, almost identical) to the
proof of Theorem 1.2 by~\cite{GS}. For completeness, we provide the proof in Appendix A.

\begin{theorem} \label{thm:degreeGame}
For every sufficiently large integer $n$ the following holds.
Suppose that:
\begin{enumerate} [$(i)$]
\item $G$ is a graph with $v(G)=n$, and
\item $b\leq \frac{\delta(G)}{4\ln n}$, and
\item $c$ is an integer such that $c(2b+1)\leq \frac{\delta(G)}{3}$,
\end{enumerate}
then, in the $(1:b)$ Maker--Breaker game played on $E(G)$, Maker has
a strategy to build a graph with minimum degree $c$. Moreover, Maker
can do so within $cn$ moves and in such a way that for every $v\in
V(G)$, as long as $d_M(v)\leq c$, we have $d_B(v)-2b\cdot d_M(v)\leq b(2\ln n+1)$ for each $v\in V(G)$.
\end{theorem}

The following theorem gives a sufficient condition for making a Hamilton-connected subgraph, which is in the basis of Maker's strategy in the fast Hamilton Cycle game in the proof of Theorem~\ref{thm:HamSmall}.
\begin{theorem}[\cite{FHK}, Proposition 2.9]
\label{thm:FHK}
Let $n$ be a sufficiently large integer and let $b\le \frac{n}{\ln^2 n}$. If $G$ is a graph with $n$ vertices whose minimum degree $\delta$ is at least $n-g(n)$, where $g(n)=o(n/\ln n)$, then Maker can build a Hamilton-connected graph playing $(1:b)$ game on $E(G)$ in \break $O(n\ln^2 n)$ moves.
\end{theorem}

In the proof of Theorem~\ref{thm:HamSmall} we use Hajnal-Szemer\'edi Theorem.

\begin{theorem}[\cite{HaSze}, Hajnal-Szemer\'edi Theorem]
\label{thm:HaSze}
If $G$ is a graph with $n$ vertices and maximum degree $\Delta$, then $G$ can be colored with $\Delta+1$ colors and moreover, each color class is of size either $\left\lfloor \frac{n}{\Delta+1} \right\rfloor$ or $\left\lceil \frac{n}{\Delta+1} \right\rceil$.
\end{theorem}

We also need the following adaptation of Theorem~\ref{thm:Krivi}. As its proof is closely following the lines of the proof of~\cite[Theorem 1]{K}, we give it in Appendix B.

\begin{theorem} [the Hamilton Cycle game] \label{thm:HamNotFast}
For every $\varepsilon>0$ there exists $\delta>0$ and an integer
$n_0:=n_0(\delta,\varepsilon)$ such that the following holds.
Suppose that:

\begin{enumerate}[$(i)$]
\item $H$ is a graph with $v(H)=n\geq n_0$, and
\item $\Delta(H)\leq \delta n$, and
\item $e(H)\leq \frac{n^2}{\ln n},$
\end{enumerate}
then for every $b\leq \left(1-\varepsilon\right)\frac{n}{\ln n}$, in
the $(1:b)$ Maker--Breaker game played on $E(K_n\setminus H)$, Maker
can build a Hamilton cycle in $O(n)$ moves.
\end{theorem}

Finally, the following lemma is used to prove Theorem~\ref{thm:HamVerySmall}.

\begin{lemma}
\label{avgDegLem}
Let $G$ be a graph on $n$ vertices whose average degree is $D < n-1$. Then, there exist two nonadjacent vertices $\{x,y\}\in V(G)$ s.t.\ $d(x)+d(y)\geq D$.
\end{lemma}

\begin{proof}
Let $m=e(G)$. We know that $m=\frac{Dn}{2}$. As $D<n-1$, a pair of non-adjacent vertices exists.  Suppose for a contradiction that for all pairs of nonadjacent vertices $\{p,q\}\in \binom{V(G)}{2} \setminus E(G)$ we have $d(p)+d(q)<D$. Then
\begin{equation}
\label{neqDeg}
\sum_{\{p,q\}\in \binom{V(G)}{2} \setminus E(G)} d(p)+d(q)<\left({n \choose 2}-m\right) D.
\end{equation}
The left side of the inequality (\ref{neqDeg}) can be written as
\begin{align}
\sum_{\{p,q\}\in \binom{V(G)}{2} \setminus E(G)} d(p)+d(q) &= \sum_{p\in V(G)} d(p)\cdot (n-1-d(p)) \nonumber \\
 &=(n-1)\sum_{p \in V(G)} d(p) - \sum_{p \in V(G)} d(p)^2. \label{csch}
\end{align}
Combining (\ref{csch}) and (\ref{neqDeg}), applying the bound $\sum_{p \in V(G)}d(p)^2 \leq m\left(\frac{2m}{n-1}+n-2\right)$ given by~\cite{DC98}, as well as the fact that $\sum_{p\in V(G)} d(p) = 2m$,
we get
\begin{displaymath}
(n-1)2m - m\left(\frac{2m}{n-1}+n-2\right) < \left({n \choose 2}-m\right) D.
\end{displaymath}
After expansion and rearrangement, using $2m = Dn$, we get that the above inequality is equivalent to $ n-1 < D$,
which is obviously in contradiction with $D<n-1$.
\end{proof}

\section{Proof of Theorem \ref{thm:perfM}}
\label{sec:makerPM}


For the proof of Theorem~\ref{thm:perfM}, we impose several restrictions on $b$. Eventually, our result in Theorem~\ref{thm:perfM} holds for $b$ as large as the threshold bias $n/\ln n$ multiplied with a small constant. As we cannot get this constant close to $1$, we make no particular effort to calculate its value or optimize it.

Let $0<\delta<1/2$ be a small positive constant and let $n_0:=n_0(\delta)$ be a positive integer as obtained by Theorem
\ref{thm:HamNotFast}, applied with $\varepsilon=99/100$. Let $b\leq \frac{\delta
n}{100\ln n}$ and let $n_1$ be such that $b=\frac{n_1}{100\ln n_1}$.
Set $m:=\frac{n_0+n_1}{\delta}$.

First we describe a strategy for Maker, that we denote by $S_M$, and then
prove it is a winning strategy. At any point throughout the game, if
Maker is unable to follow the proposed strategy $S_M$, then she forfeits the game.

Maker's strategy $S_M$ is divided into the following two stages.

\textbf{Stage 1.} In this stage, Maker will build a matching
$M'\subseteq E(K_n)$ of size $\ell:=\frac {n-m}{2}$ in $|M'|$ moves.
For each $1\leq i\leq \ell$, after Maker's $i$th move his graph
consists of a partial matching $M_i\subseteq E(K_n)$ (with
$M'=M_\ell$) and a set of isolated vertices $U_i\subseteq V(K_n)$,
where
$U_i=\{V(K_n)\setminus V(M_i)\}.$ Initially, $M_0:=\emptyset$ and $U_0:=V(K_n)$. Now, for each $1\leq
i \leq \ell$, in her $i$th move Maker will claim an arbitrary free edge
$\{v_i,w_i\}\in E(K_n[U_{i-1}])$ such that:

\begin{enumerate} [(i)]
\item $d_B(v_i,U_{i-1})=
\max\{d_B(v,U_{i-1}): v\in U_{i-1}\}$, and
\item $d_B(w_i,U_{i-1})=
\max\{d_B(w,U_{i-1}):w\in U_{i-1} \text{ and } \{w,v_i\} \text{ is free}\}$.
\end{enumerate}
As soon as stage 1 ends, Maker proceeds to stage 2.

\textbf{Stage 2.} In this stage, Maker claims a Hamilton cycle on
$E\left((K_n\setminus B)[U_\ell]\right)$, in the way provided by Theorem~\ref{thm:HamNotFast}. Moreover, Maker does so in $O(b\ln b)$ moves.

The following lemma guarantees that, for each $1\leq i\leq \ell$, Maker can always make her $i$th move according to the strategy proposed in stage 1 of $S_M$.

For $0\leq i \leq \ell$, let
$\displaystyle S_i=\sum_{v\in U_i} d_B(v,U_i)$ and
$D_i=\frac{S_i}{|U_i|}$ denote the sum and the average of the
degrees of vertices in $B[U_i]$ before Breaker's $(i+1)$st move,
respectively, and let $\displaystyle \Delta_i=\Delta(B[U_i])$.
In order to show that Maker can indeed play as proposed by $S_M$ in stage 1, it is enough to prove that $\Delta_i\leq \delta
|U_i|$ holds for each such $i$. This follows from the following lemma, which is central in proving Theorem~\ref{thm:perfM}.

\begin{lemma}\label{claim1}
If Maker can follow strategy $S_M$, then the following holds for each $1\leq i\leq \ell$:
\begin{enumerate} [$(i)$]
\item $D_i\leq 2b$, and
\item $\Delta_i\leq \delta |U_i|$.
\end{enumerate}
\end{lemma}

\begin{proof}
\begin{enumerate}[$(i)$]
\item
We prove by induction on $i$ that $D_i\leq 2b$ for each $0\leq i\leq \ell$.

For $i=0$ we trivially have that $D_i=0\leq 2b$.
Assume that Maker could follow her strategy for the first $i$ moves, and that $D_i\leq 2b$ holds. We want to show that $D_{i+1}\leq 2b$ holds as well.

Notice that since Breaker's bias is $b$, it follows that for each $i$, in his $(i+1)$st move Breaker can increase $S_i$ by at most $2b$. Moreover, playing according to the proposed strategy for stage 1, by claiming the edge $\{v_{i+1},w_{i+1}\}$, Maker decreases
$S_i$ by
\begin{displaymath}
2d_B(v_{i+1},U_i)+2d_B(w_{i+1},U_i)=2\Delta_i+2d_B(w_{i+1},U_i).
\end{displaymath}
Therefore, we have that
\begin{align}
D_{i+1}& \leq  \frac{D_i|U_i|+2b-2\Delta_i-2d_B(w_{i+1},U_i)}{|U_{i+1}|} \nonumber\\
&=D_i+2\cdot\frac{D_i+b-\Delta_i-d_B(w_{i+1},U_i)}{|U_i|-2}.
\label{eqnUpperBoundAverage}
\end{align}

To prove that $D_{i+1}\leq 2b$ holds, we distinguish between the following two cases.

\emph{Case 1:} $D_i\leq 3b/2$. In this case, using the estimate
\eqref{eqnUpperBoundAverage} and the fact that $|U_i|\geq 12$ we
have that $D_{i+1}\leq 3b/2+\frac{5b}{12-2}=2b$ as desired.

\emph{Case 2:} $D_i>3b/2$. Notice that from
\eqref{eqnUpperBoundAverage} it is enough to show that
\begin{align}
\Delta_i+d_B(w_{i+1},U_i)\geq D_i+b \label{eqnEnough}.
\end{align}
Indeed, if it is true then we obtain that $D_{i+1}\leq D_i$ which by
the induction hypothesis is bounded by $2b$.

If $\Delta_i\geq 3b$, then \eqref{eqnEnough} trivially holds, as
$\Delta_i+d_B(w_{i+1},U_i)\geq \Delta_i\geq 2b+b\geq D_i+b$. Otherwise, we
have that $3b/2<D_i\leq \Delta_i<3b$. Let $x$ be the number of vertices
in $U_i$ with degree at least $b$ in Breaker's graph. Notice that
since $3b/2< D_i\leq \frac{3bx+(|U_i|-x)b}{|U_i|}$, it follows
that $x> \frac{|U_i|}{4}$. Now, since $\frac{|U_i|}{4}>3b$, it
follows that there exists a vertex $w\in U_i$ for which the edge
$\{v_{i+1},w\}$ is free and $d_B(w,U_i)\geq b$. Therefore,
$d_B(w_{i+1},U_i)\geq b$. Finally, combining it with the fact that
$\Delta_i\geq D_i$, we conclude that \eqref{eqnEnough} holds.

\item Notice first that while $\delta|U_i|\geq 2b(1+2\ln n)$, the claim is
true as a consequence of Theorem~\ref{thm:degreeGame}, as the strategies of Maker in both games are the same: to touch the vertex of the largest degree. The conditions of Theorem~\ref{thm:degreeGame} are satisfied as $c=1$ which satisfies condition $(iii)$, and $\delta|U_i|\geq 2b(1+2\ln n)$ implies that $b\leq \frac{\delta |U_i|}{2(1+2\ln n)}\leq \frac{|U_i|(1-\delta)}{4\ln |U_i|}$, which satisfies $(ii)$. Therefore, it is enough to prove the lemma for $i$'s such that
$|U_i|<\frac{2b(1+2\ln n)}{\delta}\leq \frac{n}{10}$.

Now, let us look at the case when $\delta |U_i|< 2b(1+2\ln n)$. Let $s=n-\frac{b(1+2\ln n)}{\delta}$. Assume towards a contradiction
that for some $s\leq i_0\leq \ell$, after Maker's $i_0$th move,
there exists a vertex $v\in U_{i_0}$ for which
$d^*:=d_B(v,U_{i_0})>\delta |U_{i_0}|$. Now, for each $k\geq 1$ we will
recursively construct a set $R_k$ for which the following holds:
\begin{enumerate} [(a)]
\item $R_k\subseteq U_{i_0-k}$,
\item $|R_k|=k+1$, and
\item for each $k\geq 1$, after $(i_0-k)$th round,
\begin{align} \sum_{u\in
R_k}d_B(u,U_{i_0-k})\geq (k+1)\left(d^*-2b\cdot \sum_{j=2}^{k+1}\frac
1j\right). \label{eqnInduction}
\end{align}
\end{enumerate}

For $k=1$, let $R_{1}:=\{v,v_{i_0}\}\subseteq U_{i_0-1}$, where $v$ is
a vertex with $d_B(v,U_{i_0})=d^*$ and $v_{i_0}\in U_{i_0-1}$ is the
vertex that Maker has touched in her $i_0$th move. Since
$d_B(v_{i_0},U_{i_0-1})=d_{i_0}\geq d^*$ and Breaker, claiming $b$ edges per move, could not increase the degrees of these two vertices by more than $2b$ in his $i_0$th move, inequality
\eqref{eqnInduction} trivially holds.

Assume we built $R_{k}$, satisfying $(a),(b)$ and $(c)$, we want to build $R_{k+1}$. Let
$v_{i_0-k}\in U_{i_0-k-1}$ be the vertex that Maker has touched in
her $(i_0-k)$th move. Notice that $v_{i_0-k}\notin R_k$ (otherwise $R_k$
cannot be a subset of $U_{i_0-k}$) and $U_{i_0-k}\subseteq U_{i_0-k-1}$.

Hence, we conclude that before Maker's $(i_0-k)$th move
\begin{align}
d_B(v_{i_0-k},U_{i_0-k-1})&\geq \frac{1}{|R_k|}\sum_{u\in
R_k}d_B(u,U_{i_0-k})
= \frac{1}{k+1}\sum_{u\in R_k}d_B(u,U_{i_0-k}).
\label{eqnInductionStepEstimate}
\end{align}

Define $R_{k+1}:=R_k\cup \{v_{i_0-k}\}$. We have that $R_{k}\subseteq U_{i_0-k}\subseteq U_{i_0-k-1}$ and $v_{i_0-k}\in U_{i_0-k-1}$, which together imply that $R_{k+1}\subseteq U_{i_0-k-1}$, satisfying $(a)$. Also, $|R_{k+1}|=|R_{k}|+1=k+2$ satisfies $(b)$. Combining
\eqref{eqnInduction}, \eqref{eqnInductionStepEstimate} and the fact that Breaker can increase the sum of all degrees in $U_{i_0-k-1}$ by at most $2b$ in one move we obtain that

\begin{align}
\sum_{u\in R_{k+1}}d_B(u,U_{i_0-k-1})&\geq \sum_{u\in
R_k}d_B(u,U_{i_0-k}) + \frac{1}{k+1}\sum_{u\in
R_k}d_B(u,U_{i_0-k})-2b\nonumber \\
&=\frac{k+2}{k+1}\cdot\sum_{u\in
R_k}d_B(u,U_{i_0-k})-2b \nonumber \\
&\geq (k+2)\cdot \left( \frac{1}{k+1}\cdot (k+1)\left(d^*-2b\cdot
\sum_{j=2}^{k+1}\frac 1j\right)-\frac{2b}{k+2}  \right) \nonumber
\\
&= (k+2)\left(d^*-2b\cdot \sum_{j=2}^{k+2}\frac 1j\right), \nonumber
\end{align}
and so the property $(c)$ is also satisfied for $R_{k+1}$. This completes the inductive step.
\allowdisplaybreaks
Now, for $k=|U_{i_0}|-1$ we obtain that
\begin{align}
D_{i_0-k}&=\frac{\displaystyle \sum_{u\in U_{i_0-k}}d_B(u,U_{i_0-k})}{|U_{i_0-k}|} \geq \frac{\displaystyle \sum_{u\in
R_k}d_B(u,U_{i_0-k})}{|U_{i_0-k}|}\nonumber \\
&\geq \frac{(k+1)\left(d^*-2b\cdot \displaystyle \sum_{j=2}^{k+1}\frac
1j\right)}{3k+1} \geq \frac{d^*-2b\ln |U_{i_0}|}{3} \nonumber \\
&\geq \frac{\delta|U_{i_0}|-2b\ln |U_{i_0}|}{3} >2b, \nonumber
\end{align}
which is clearly in contradiction with $(i)$. This completes the proof
of Lemma~\ref{claim1}.
\end{enumerate}
\end{proof}

\begin{proof}

If Maker can follow stages 1 and 2 of the proposed strategy $S_M$ without forfeiting the game then she wins, since a Hamilton cycle on $E\left((K_n\setminus B)[U_\ell]\right)$ obviously
contains a perfect matching. It thus suffices to show that indeed Maker can follow the proposed strategy without forfeiting the game. We consider each stage separately.

Maker can follow stage 1 of the proposed strategy $S_M$, as shown by Lemma~\ref{claim1}.

Moving on to stage 2, let $H=B[U_\ell]$. When stage 1 is over, Lemma~\ref{claim1} gives that $\Delta(H)\leq\delta|U_\ell|$ and $e(H)=\frac{D_\ell |U_\ell|}{2}\leq \frac{|U_\ell|^2}{\ln |U_\ell|}$. This satisfies the conditions of Theorem \ref{thm:HamNotFast}, so Maker can claim a Hamilton cycle on $V(K_n\setminus B)[U_\ell]$ in $O(|U_\ell|)=O(b\ln b)$ moves.

Note that the constant $\alpha$ in the statement of Theorem~\ref{thm:perfM} is obtained by taking $\alpha:=\frac{\delta}{100}$.
\end{proof}

\section{Proof of Theorem~\ref{thm:HamVerySmall}}
\label{sec:hamVS}

\begin{proof}
In this section we give Maker's strategy to win the Hamilton Cycle game when bias $b$ is not too large, namely $b=o\left(\frac{\ln n}{\ln \ln n}\right)$. Throughout the game, Maker will maintain a collection of paths $\cP$ in her graph. Maker's strategy is divided into five stages.

At the beginning of the game, all vertices are isolated, and $\cP$ consists of $n$ paths of length $0$, with $2n$ endpoints altogether (as each of $n$ vertices is seen as both the first and the last vertex of a path). By $End_i$ we denote the multiset of endpoints of $\cP$ after Maker's $i$th move (omitting index $i$ when it is not crucial). During the first two stages, Maker will claim only edges between the endpoints of the paths in $\cP$ (thus connecting two paths into one), so in each of those moves $|\cP|$ will be reduced by one and $|End|$ will be reduced by two.

\textbf{Stage 1.} We fix $\delta :=999/1000$. This stage lasts for $\ell:=n-30b \ln b$ moves, so at the end of the stage we will have $|End_{\ell}|=60 b\ln b$. We will later show that for every vertex $v\in End_{\ell}$ it holds that $d_B(v, End_{\ell})<\delta|End_{\ell}|+b$.

In this stage, Maker plays two games in parallel.

\begin{enumerate}[(1)]
\item In every \emph{odd} move $i$, Maker considers the Breaker's graph induced on $End_{i-1}$, claiming an arbitrary free edge $\{v_i, w_i\}\in E(K_n[End_{i-1}])$, such that:
\begin{enumerate}[(i)]
\item $d_B(v_i, End_{i-1})=\max \left\lbrace d_B(v, End_{i-1}): v\in End_{i-1}\right\rbrace$, and
\item $d_B(w_i, End_{i-1})=\max \left\lbrace d_B(w, End_{i-1}): w\in End \mbox{ and } \{w,v_i\} \mbox{ is free}\right\rbrace$.
\end{enumerate}

\item In every \emph{even} move $i$, Maker considers the \emph{total} Breaker's degree on vertices in $End_{i-1}$, claiming an arbitrary free edge $\{v_i,w_i\} \in E(K_n[End_{i-1}])$, such that:
\begin{enumerate}[(i)]
\item $d_B(v_i)=\max \left\lbrace d_B(v): v\in End_{i-1}\right\rbrace$, and
\item $d_B(w_i, End_{i-1})=\max \left\lbrace d_B(w, End_{i-1}): w\in End_{i-1} \mbox{ and } \{w,v_i\} \mbox{ is free}\right\rbrace$.
\end{enumerate}
\end{enumerate}

\textbf{Stage 2.} In every move of this stage, Maker claims an arbitrary free edge $\{v_i, w_i\}\in E(K_n[End_{i-1}])$, such that
\begin{multline*}
d_B(v_i, End_{i-1})+d_B(w_i,End_{i-1})= \\
= \max \left\lbrace d_B(v, End_{i-1})+d_B(w, End_{i-1}): \{v,w\}\subseteq End_{i-1} \mbox{ and } \{v,w\} \mbox{ is free}\right\rbrace.
\end{multline*}
She plays like this until there are no free edges within $K_n[End]$. We will later prove that this stage lasts for at least $23b\ln b$ moves.

\textbf{Stage 3.} During the course of this stage, Maker makes sure that all paths in $\cP$ are of length greater than $n^{3/4}$. In this stage, the number of paths in $\cP$ remains the same, while some of Maker's edges will be ``forgotten''. Given a path, we say that its \emph{near-middle vertices} are the 1\verb|%| of its vertices that are the closest to its middle (breaking ties arbitrarily).

For each path $P\in \cP$ of length at most $n^{3/4}$, Maker will spot the longest path $Q$ in $\cP$, and claim an edge connecting one of the endpoints of $P$ to a near-middle vertex $x$ of $Q$ which is such that for one of the neighbors $y$ of $x$ on $Q$ we have $d_B(y)<18b\ln n$. (We later show that such an edge will be available for Maker.)  Maker then ``forgets'' about the edge $\{x,y\}$, splitting $Q$ into two parts -- $Q_1$, with an endpoint $x$, and $Q_2$, with an endpoint $y$, now seeing paths $P$ and $Q_1$ joined into one path, while $Q_2$ remains a separate path, see Figure~\ref{fig:stage3}.

\begin{figure}[htb]
\begin{center}
\includegraphics[scale=1]{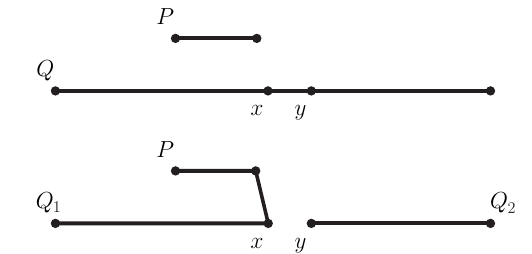}
\end{center}
\caption{Paths $P$ and $Q$ before Maker's move, and the newly obtained pair of paths}
\label{fig:stage3}
\end{figure}

Note that both of the newly obtained paths are longer than $n^{3/4}$, as $b=o\left(\frac{\ln n}{\ln \ln n}\right)$. This stage obviously lasts for less than $7b\ln b$ moves, as there are $7b\ln b$ paths at the  beginning of stage 3.

\textbf{Stage 4.} In this stage Maker will repeatedly select two arbitrary paths from $\cP$, using the method of P\'{o}sa rotations (see~\cite{Posa}) on each of them to eventually be able to connect them into one (longer) path. She will repeat the procedure until only one, Hamilton, path remains in $\cP$. This stage will be subdivided into \emph{phases}, where a new phase begins whenever Maker selects two paths, and ends when Maker connects them into a path. The inner vertices of any path $P \in \cP$ that have at least $n^{1/2}$ Breaker's edges whose other endpoint is in $V(\cP\setminus \{ P \})$ will be called \emph{saturated}.

Let us now describe in more detail the process of connecting two paths into one. We denote the paths by $P_1$ and $P_2$, with endpoints $v_{P_1}^1, v_{P_1}^2$ and $v_{P_2}^1$, $v_{P_2}^2$.
Maker alternately makes P\'{o}sa rotations on $P_1$ and $P_2$ by doing the following.

In her $m$th move of the phase, where $1\leq m \leq 2b+1$, if $m = 2i - 1$ is odd, Maker spots a pair of consecutive vertices $(x_i,x_i')$ on $P_1$ such that $x_i$ is between $v_{P_1}^1$ and $x_i'$ and satisfying conditions that we will describe below, and she claims the edge $\{v_{P_1}^1,x_i'\}$. If $m = 2i$ is even, Maker spots a pair of consecutive vertices $(y_i,y_i')$ on $P_2$ such that $y_i$ is between $v_{P_2}^1$ and $y_i'$ and satisfying conditions that we will describe below, and she claims the edge $\{v_{P_2}^1,y_i'\}$.

When spotting $(x_i,x_i')$ on $P_1$ in odd moves, Maker makes sure that $x_i$ is not saturated, the edges $\{v_{P_1}^1,x_i'\}$ and $\{v_{P_2}^1,x_i\}$ are free, and also, for all $y_j \in V(P_2)$, $1\leq j < i$, the edges $\{x_i, y_j\}$ are free. Note that by claiming the edge $\{v_{P_1}^1,x_i'\}$, Maker has created a new path (a P\'{o}sa rotation of $P_1$) that has an endpoint $x_i$.

When spotting $(y_i,y_i')$ on $P_2$ in even moves, Maker makes sure that $y_i$ is not saturated, the edges $\{v_{P_2}^1,y_i'\}$ and $\{v_{P_1}^1,y_i\}$ are free, and also, for all $x_j \in V(P_1)$, $1\leq j\leq i$, the edges $\{x_j, y_i\}$ are free. Note that by claiming the edge $\{v_{P_2}^1,y_i'\}$ Maker has created a new path (a P\'{o}sa rotation of $P_2$) that has an endpoint $y_i$.

Before each of the odd moves $m = 2i + 1$, Maker checks whether any of the vertices\linebreak[4] $v_{P_1}^1, x_1, \dots, x_i$ can be connected to any of the vertices $v_{P_2}^1, y_1, \dots, y_i$. As soon as this is possible, Maker connects them, thus connecting the two paths (actually, their P\'{o}sa rotations) and finishes the phase.

A phase can last for up to $2b+1$ moves, as in her $m$th move, for $m$ even, Maker creates $m/2+1$ new threats which Breaker needs to claim immediately if she wants to prevent the phase from ending. Thus, knowing that there are $7b\ln b$ paths at the beginning of stage~4, the whole stage lasts for less than $14 b \ln b\cdot (b+1)$ moves.

\textbf{Stage 5.} Maker completes the Hamilton cycle by repeatedly making P\'{o}sa rotations, in a similar way as she did in each phase of stage~4. The only difference is that Maker, in her mind, halves the path into two halves -- left and right, and the P\'{o}sa rotations are alternately performed -- in odd moves from the left endpoint to the left half, and in even moves from the right endpoint to the right half. The whole process then continues in analogous fashion, and this stage lasts for at most $2b + 1$ moves, until the point when Maker closes her paths into Hamilton cycle.

\medskip

If Maker can follow the given five-stage strategy, she obviously wins the game within the required number of moves. We will show that Maker can follow the proposed strategy, and that the move count for each stage matches the desired one. We perform the analysis for each stage separately.

For $0\leq i \leq \ell$, we let $S_i=\sum\limits_{v\in End_i} d_B(v,End_i)$ and $D_i=\frac{S_i}{|End_i|}$ denote the sum and the average of the degrees of vertices in $B[End_i]$ before Breaker's $(i+1)$st move respectively, and let $\Delta_i=\Delta(B[End_i])$.

\medskip

\textbf{Stage 1.} \label{box} We will first look at the game $(2)$ that Maker plays in every even move of hers. This part of the strategy can be analyzed as playing an auxiliary Box game, where Maker takes the role of BoxBreaker -- each $v\in End$ corresponds to one box whose elements are all edges incident to $v$.

With every edge $\{p,q\}$ that Breaker claims, we imagine that BoxMaker claims an element from box $p$ and an element from box $q$ (two elements in total). Note that the same vertex can account for two boxes if it is the double endpoint of a path of length 0, so for each edge claimed by Breaker we have up to four moves that BoxMaker plays in the Box game. As Maker plays this game in every second move, and Breaker claims $2b$ edges in two moves, there are up to $8b$ elements claimed by BoxMaker in each move.

Hence, to show that Maker (alias BoxBreaker) has an upper hand in this game, having in mind that Maker--Breaker games are bias-monotone (see~\cite{HKSSbook}) we will look at the game \break $B(2n, 2n^2, 8b, 1)$.

Now, we want to estimate the size of the largest box that BoxMaker could fill until the end of the game. Note that this gives us the maximum degree in Breaker's graph at every $v\in End$, at any point of stage 1. The size of the largest box is
\begin{align}
\allowdisplaybreaks
s&=\frac{8b}{2n}+\frac{8b}{2n-1}+\cdots +\frac{8b}{1} =8b\sum_{j=1}^{2n}\frac{1}{j} \le 8b\ln (2n) <16 b\ln n. \label{eqBox}
\end{align}
This implies that when stage 1 is over, every vertex $v \in End$ it holds that $d_B(v)<16b\ln n$.

Before we show that Maker can play according to $(1)$, we need the following claim.
\begin{claim}
\label{claim2}
Let $e=\{v,w\}$ be the edge that Maker selects while playing according to $(2)$ in his $(i+1)$st move, where $i$ is an odd integer. It holds that $d_B(v,End_i)+d_B(w,End_i)\geq D_i$.
\end{claim}
\begin{proof}
Since $v$ is chosen so that $d_B(v)$ is maximal we do not right away have that $d_B(v, End_i)>D_i$ holds. If $d_B(v,End_i)\geq D_i$, we have nothing to prove. So, suppose $x:=d_B(v,End_i)<D_i$ and let $y:=|\{w\in End_i: d_B(w,End_i)\ge D_i-x\}|$. It holds that 
\begin{displaymath}
D_i<\frac{\Delta_i \cdot y+(|End_i|-y)\cdot (D_i-x)}{|End_i|}
\end{displaymath} 
and $D_i\le \Delta_i$. By expanding, we get
\begin{equation}
D_i|End_i|<\Delta_i \cdot y+D_i \cdot|End_i|-x|End_i|-y\cdot D_i+xy,
\end{equation}
and we can easily obtain
\begin{equation}
y>\frac{x|End_i|}{\Delta_i-D_i+x}> x.
\end{equation}
Thus, there are more than $x$ vertices in $End_i$ of degree at least $D_i-x$. So, playing according to her strategy in game $(2)$, Maker can find a vertex $w$ that is not adjacent to $v$ in Breaker's graph such that $d_B(w,End_i)\ge D_i-x$ is maximal, and claim the edge $\{v, w\}$, thus making sure that $d_B(v, End_i)+d_B(w, End_i)\geq D_i$. This proves the claim.
\end{proof}

Now, we look at part $(1)$. Maker plays this game in her odd moves, so between the two odd Maker's moves Breaker adds $b'=2b$ edges. 
In order to show that Maker can play according to the given strategy, we need to show the following claim.

\begin{claim}\label{claim3_1}
The following two properties hold for each $i=2m+1$, $m\geq 0$:
\begin{enumerate} [$(i)$]
\item $D_i\leq 2b'$, and
\item $\Delta_i< \delta |End_i|$.
\end{enumerate}
\end{claim}

\begin{proof}
The proof of this claim is very similar to the proof of Lemma~\ref{claim1}, so we will omit some details.

\begin{enumerate}[$(i)$]
\item
Note that between moves $i$ and $i+2$, Breaker can increase the value of $S_i$ by at most $2b'$, so $S_{i+2}\leq S_i+2b'$. On the other hand, playing on $End$ according to the proposed strategy, after claiming the edge $\{v_{i+2}, w_{i+2}\}$, Maker decreases $S_i$ by
\begin{displaymath}
2d_B(v_{i+2}, End_{i+1})+2d_B(w_{i+2}, End_{i+1}) \geq 2\Delta_i+2d_B(w_{i+2},End_{i+1}).
\end{displaymath}
Also, according to Claim~\ref{claim2}, Maker in her even move $i+1$ claims the edge $\{p,q\}$ such that $d_B(p,End_i)+d_B(q,End_i)\geq D_i$, thus decreasing $S_i$ by at least $2D_i$.\\
Therefore, we have that
\begin{align}
D_{i+2} & \leq \frac{D_i(|End_i|)+2b'-2D_i-2\Delta_i-2d_B(w_{i+2},End_{i+1})}{|End_{i+2}|}\nonumber \\
&\leq \frac{D_i(|End_i|-4+4)+2b'-2D_i-2\Delta_i-2d_B(w_{i+2},End_{i+1})}{|End_{i+2}|} \nonumber \\
&= D_i+2\frac{D_i+b'-\Delta_i-d_B(w_{i+2},End_{i+1})}{|End_i|-4}. \label{hamUpperAvg}
\end{align}
Note that the equality (\ref{hamUpperAvg}) has the same structure as (\ref{eqnUpperBoundAverage}), so we can follow the rest of the proof of Lemma~\ref{claim1} $(i)$
line by line and show by induction on $i$ that for all odd moves after Maker's move it holds that $D_i\leq 2b'$. Here, instead of $U_i$, we use $End_i$ to denote the multiset of endpoints, we have $|End|>14$, and Breaker's bias is $b'$.

\item
Just like in the proof of $(i)$, we also rely on the proof of Lemma~\ref{claim1} $(ii)$, and omit some of the details. While $|End_i|>n/10$, it holds that $\delta|End_i| > 2b'(1+2 \ln 2n)$ and $b'<\frac{(1-\delta)|End_i|}{4\ln|End_i|}$, so the claim is true as a consequence of Theorem~\ref{thm:degreeGame}. Therefore, we need to prove the claim for the values of $i$ such that $|End_i|\leq \frac{n}{10}$. 

Analogously to the proof of Lemma~\ref{claim1}, we let $s=n-\frac{n}{20}$. Towards a contradiction, we assume that for some $s\leq i_0\leq \ell$ and $i_0=2m+1$, $m\geq 0$, after Maker's $i_0$th move there exists a vertex $v \in End_{i_0}$ such that $d^*:=d_B(v,End_{i_0})\geq \delta|End_{i_0}|$. We also inductively construct sets $R_k$, for $k\geq 1$, exactly the same as in the proof of Lemma~\ref{claim1}, with the following properties:
\begin{enumerate} [(a)]
\item $R_k\subseteq End_{i_0-2k}$,
\item $|R_k|=k+1$, and
\item for each $k\geq 1$, after $(i_0-2k)$th round,
\begin{align} \sum_{u\in
R_k}d_B(u,End_{i_0-2k})\geq (k+1)\left(d^*-2b'\cdot \sum_{j=2}^{k+1}\frac
1j\right). \nonumber
\end{align}
\end{enumerate}

To finish the proof, suppose that $k=|End_{i_0}|-1$. We obtain
\begin{align}
D_{i_0-2k}&=\frac{\displaystyle \sum_{u\in End_{i_0-2k}}d_B(u,End_{i_0-2k})}{|End_{i_0-2k}|}\geq \frac{\displaystyle \sum_{u\in
R_k}d_B(u,End_{i_0-2k})}{|End_{i_0-2k}|}\nonumber \\
&\geq \frac{(k+1)\left(d^*-2b'\cdot \displaystyle \sum_{j=2}^{k+1}\frac
1j\right)}{5k+1} \geq \frac{d^*-2b'\ln |End_{i_0}|}{5} \nonumber \\
&\geq \frac{\delta|End_{i_0}|-2b'\ln |End_{i_0}|}{5} >2b', \nonumber
\end{align}

contradicting $(i)$. This completes the proof of Claim~\ref{claim3_1}.
\end{enumerate}
\end{proof}

To complete the proof that Maker can play according to the proposed strategy for stage 1, we can use Claim~\ref{claim3_1} $(ii)$, and observe that if $\ell$ is odd, then for every $v\in End$ we have $d_B(v,End)<\delta|End|<\delta|End|+b$. Otherwise, after Maker's $(l-1)$st move for every $v\in End$ it holds that $d_B(v,End)<\delta|End|$. In the consequent move, Breaker can add at most $b$ edges to some vertex, and so, after Maker's $l$th move played according to strategy in $(2)$, $d_B(v,End)<\delta|End|+b$.

\medskip

\noindent \textbf{Stage 2.} Before the first move, that we denote by $m$, of Breaker in this stage, using Claim~\ref{claim3_1} and knowing that Breaker possibly played one move after the last odd move of Maker in stage~1, we have $D_{m-1}\leq 4b+1$ and $\Delta_{m-1}<\delta|End_{m-1}|+b$.

To show that Maker can play as proposed, let us observe that this stage will certainly be played while $D<|End|-1-\frac{2b}{|End|}$ (as in that case there are free edges among the endpoints of the paths).
In every move $i$ in this stage, Breaker adds at most $b$ edges to $B[End]$, and thus increases $S_{i-1}$ by at most $2b$. On the other hand, Maker playing as proposed and claiming the edge $\{v_i,w_i\}$ decreases $S_{i-1}$ by $d_B(v_i, End_{i-1})+d_B(w_i,End_{i-1})$. Thus, in a way similar to $(\ref{hamUpperAvg})$ we obtain that
\begin{equation}
D_i\leq D_{i-1}+2\frac{D_{i-1}+b-d_B(v_i, End_{i-1})-d_B(w_i,End_{i-1})}{|End_{i-1}|-2}.
\end{equation}
If $d_B(v_i, End_{i-1})+d_B(w_i,End_{i-1}) \geq D_{i-1}+b$, then average Breaker's degree does not increase. Otherwise, as a consequence of Lemma~\ref{avgDegLem}, where $G=B[End]$, we get that $d_B(v_i, End_{i-1})+ \break d_B(w_i,End_{i-1}) \geq D_{i-1}$, as long as there are still free edges among vertices in $End$. So, the average Breaker's degree can increase by at most $\frac{2b}{|End|-2}$.

Let $y:=|End|=60b\ln b$ at the beginning of stage 2, and $t$ be the number of moves played in this stage. During $t$ moves the average degree increases by at most
\begin{equation}
\label{stg2DegInc}
\sum_{j=1}^{t}\frac{2b}{y-2j}= 2b\sum_{j=1}^{t}\frac{1}{y-2j}\leq 2b\ln \frac{y}{y-2t}.
\end{equation}
Altogether, the average degree after $t$ moves is at most $4b+1+2b\ln \frac{y}{y-2t}$. Having this in mind, and that the stage will last as long as $D<|End|-1-\frac{2b}{|End|}$, as well as that $|End| = y-2t$ holds at the end of stage 2, we get that
$t\geq 23b\ln b$ holds.

Note that after this stage is over, the overall maximal degree in Breaker's graph could have increased by at most $23b^2\ln b<b\ln n$, knowing that $b=o\left(\frac{\ln n}{\ln \ln n}\right)$. So, for all $x\in End$ it holds that $d_B(x)<17b\ln n$, which will come handy in the analysis of stage 3.

\medskip

\textbf{Stage 3.} Maker's graph so far consists of disjoint paths.
If Maker can follow the proposed strategy, stage~3 lasts for less than $7b\ln b$ moves and thus Breaker claims less than $7b^2\ln b$ edges. So, for all $x\in End$ after this stage is over, we will have
\begin{equation}
\label{neq:st3}
d_B(x)<17b\ln n + 7b^2\ln b<18b\ln n.
\end{equation}

Now, we show that it is indeed possible for Maker to complete this stage. As the average length of a path in Maker's graph is $2n/|End|=\Theta\left(\frac{n}{b\ln b}\right)\gg n^{3/4}$, the set of near-middle vertices of the longest path is of size $\Theta\left(\frac{n}{b\ln b}\right)$. This, together with the fact that $d_B(x)<18b\ln n$ for all $x \in End$, gives that it is always possible to find a free edge between one endpoint $u$ of $P$ and some vertex from the set of near-middle vertices, as there are at least $\frac{n}{100|End|}-18b\ln n \geq \frac{Cn}{b\ln b}$ vertices nonadjacent to $u$, for some positive constant $C$. 
For a contradiction, suppose that for each vertex $v$ of them it holds that $d_B(v)\geq 18b\ln n$. But then we obtain that Breaker had to claim so far at least $\frac{Cn}{b\ln b}\cdot 18b\ln n > \frac{C_1n\ln n}{\ln b}$ edges. This is in contradiction with the estimated number of edges that Breaker could claim till the end of this stage.

\medskip

\noindent \textbf{Stage 4.} If Maker can follow the proposed strategy in this stage, the number of edges that Breaker claims during stage~4 is less than
\begin{equation}
\label{neq:st4}
14b^2\ln b(b+1)<28b^3\ln b.
\end{equation}
All paths in Maker's graph are longer than $n^{3/4}$. The number of saturated vertices is $o(v(P))$, as otherwise Breaker must have claimed at least $c\cdot n^{1/2}\cdot (n^{3/4}-1)=\Theta(n^{5/4})$, for some positive constant $c$, which is in contradiction with the number of edges that Breaker could claim in the whole game.

Each phase of stage~4 ends after at most $2b+1$ moves of Maker. Indeed, at latest when the pair of vertices $(y_b, y_b')$ is chosen, $y_b$ has no neighbors among vertices $v_{P_1}^1, x_1, x_2, \dots, x_b$, and so after Breaker's move, there has to be at least one free edge among the pairs of vertices \linebreak $\lbrace (y_b,v_{P_1}^1), (y_b,x_1), (y_b,x_2), \dots , (y_b,x_b)\rbrace$, and therefore in her consequent move, Maker completes the phase.

Now, we need to show that in each phase in this stage, Maker can choose the pairs of vertices as described above.

When each of the $7b\ln b-1$ phases begins, using (\ref{neq:st3}) and (\ref{neq:st4}), we get that for each of the endpoints $v$ of the two selected paths it holds that $d_B(v)<18b\ln n+28b^3\ln b<30b^2\ln n=o(n^{3/4})$. The number of saturated vertices on each path $P$ is $o(v(P))$ when stage 4 begins and after adding at most $28b^3\ln b=o\left(\ln^3 n\right)$ additional edges it is still $o(v(P))$, as $v(P)\geq n^{3/4}$. Also, for each $x_i$ and $y_i$, we know that $d_B(x_i, P_2)=O(n^{1/2})$ and $d_B(y_i,P_1)=O(n^{1/2})$. So, for each move $m$, $1\leq m \leq 2b+1$, if $m=2i-1$ is odd, the number of choices for the pairs $(x_i, x_i')$ is $v(P)-3-d_B(v_{P_1}^1)-d_B(v_{P_2}^1)-o(v(P))-\sum_{j=1}^{i-1}d_B(y_j, P_1)>n^{3/4}-3-2\cdot 30b^2\ln n-o(n^{3/4})>(1-o(1))n^{3/4}$. The same calculation applies in the move $m=2i$, for the pairs of vertices $(y_i, y_i')$, $1\leq m \leq 2b$. So, there are enough pairs to choose from. Note that when this stage ends, there is a Hamilton path whose endpoints have degree $d_B(v)<18b\ln n+28b^3\ln b<30b^2\ln n$.

\medskip

\noindent \textbf{Stage 5.}
We know that after stage 4 it holds that $d_B(v_{P}^1)<30b^2\ln n$ and $d_B(v_{P}^2)<30b^2\ln n$. In this stage, provided that Maker can follow it, Breaker can claim at most $2b^2+2b$ edges in total, and so by the end of this stage it holds that $d_B(v_{P}^1)<30b^2\ln n+2b^2+2b<32b^2\ln n$ and $d_B(v_{P}^2)<30b^2\ln n+2b^2+2b<32b^2\ln n$. With reasoning and calculation similar to the one for a phase in stage~4, we get that Maker can have more than $(1/2-o(1))n$ choices for each pair $(x_i, x_i')$, respectively $(y_i, y_i')$, and so she can complete this stage in the proposed time and finish the game.
\end{proof}

\section{Proof of Theorem~\ref{thm:HamSmall}}
\label{sec:hamS}

\begin{proof}
Maker's strategy in the Hamilton Cycle game is divided into three main stages.

\noindent \textbf{Stage 1.} Maker splits the vertices of the board into two sets, $X$ and $I$, such that at the beginning $X=\emptyset$ and $I=V$. Throughout  stage 1, the set $X$ will contain the vertices of vertex disjoint Hamilton-connected subgraphs in Maker's graph. Maker will build each such subgraph one after another, and at any point only one such subgraph is being built, while the others are completed. Maker's graph on $I$ will be a collection of paths, each of length $\geq 0$, denoted by $\cP$, with the set of endpoints denoted by $End(\cP)$. Note that isolated vertices in $\cP$ (viewed as paths of length $0$) appear twice in $End(\cP)$. Both $\cP$ and $End(\cP)$ are updated dynamically. At the beginning, every $v\in I$ is considered as a path of length 0.

During this stage, Maker plays the following two games in parallel.
\begin{enumerate} [$(1)$]
\item In her odd moves, Maker builds $L=L(b,n):=13b\ln n$ Hamilton-con\-ne\-cted subgraphs of order $t=t(b,n):=\frac{1}{2}b\ln^2n$ that are vertex disjoint. She builds them one by one, repeatedly choosing new $t$ isolated vertices from $I$ that are independent in Breaker's graph and moving them to $X$, whenever the previous Hamilton-connected subgraph is completed.
\item In each of her even moves, Maker chooses a vertex $v\in End(\cP)$, $v=v_P^1$ s.t.\ 
\begin{math}
 \displaystyle d_B(v) = \break  \max_{w\in End(\cP)} d_B(w)
\end{math}
 (ties broken arbitrarily) and claims a free edge between $v$ and some other vertex $u \in End(\cP)\setminus \{v\}$, $u\neq v_P^2$.
\end{enumerate}

By Theorem~\ref{thm:FHK}, Maker needs $O(t\ln^2 t)$ moves to build one Hamilton-connected subgraph of order $t$, thus for $L$ such subgraph she needs $O(Lt\ln^2 t)=O(b^2\ln^5 n)$ moves. So, this stage lasts $Cb^2\ln^5 n$, $C>0$ moves.

\begin{figure}[htb]
\begin{center}
\includegraphics[scale=0.8]{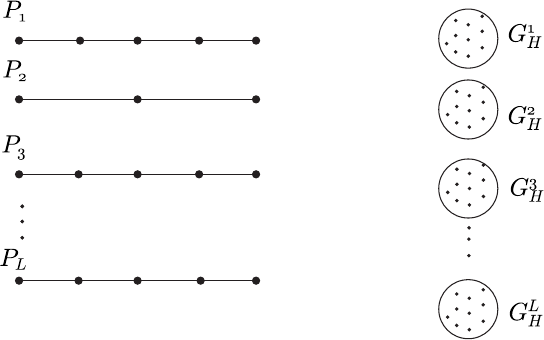}
\end{center}
\caption{Maker's graph at the end of stage 2}
\label{fig:st2}
\end{figure}

\begin{figure}[htb]
\begin{center}
\includegraphics[scale=0.8]{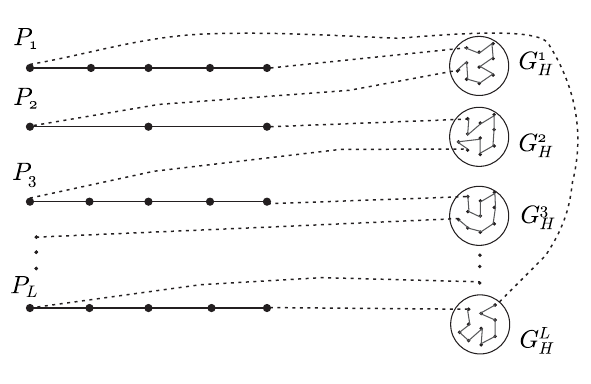}
\end{center}
\caption{Maker's strategy in stage 3}
\label{fig:st3}
\end{figure}

\noindent \textbf{Stage 2.} In this stage, with each of her edges, Maker plays in the same way as in part (2) of stage~1. In each move, she chooses a vertex $v\in End(\cP)$, $v=v_P^1$ s.t. $\displaystyle d_B(v)=\max_{w\in End(\cP)}d_B(w)$ (ties broken arbitrarily) and claims a free edge between $v$ and some other vertex $u \in End(\cP)\setminus\{v\}$, $u\neq v_P^2$.
She plays like this until $|\cP|=L$. This stage lasts $n-|X|-L-\frac{C}{2}b^2\ln^5 n=n-Lt-L-\frac{C}{2}b^2\ln^5 n=n-13b\ln n\left(\frac{1}{2}b\ln^2 n+1\right)-\frac{C}{2}b^2\ln^5 n$ moves.

\noindent \textbf{Stage 3.} At the beginning of this stage there are $L$ Hamilton-con\-nect\-ed subgraphs,\linebreak[4] $G_H^1, G_H^2, \dots, G_H^L$, and $L$ paths, $P_1,P_2,\dots, P_L$, see Figure~\ref{fig:st2}. Before this stage begins, Maker fixes which paths (through which exact endpoints) will be joined to specific Hamilton-con\-nect\-ed subgraphs. She uses the following rule: for each $i$, $1\le i \le L$, $v_{P_i}^1$ and $v_{P_{(i \mod L)+1}}^2$ will be connected to two arbitrary different vertices in $G_H^i$, see Figure~\ref{fig:st3}. In each move, Maker chooses a vertex $v \in End(\cP)$ such that $\displaystyle d_B(v)\ge \max_{w \in End(\cP)}d_B(w)$ (ties broken arbitrarily), connects it to the Hamilton-connected subgraph according to the aforementioned rule, and removes $v$ from $End(\cP)$. In order to do so, Maker plays an auxiliary Box game, pretending to be BoxBreaker. Since Maker has to connect two endpoints to the distinct vertices of some Hamilton-connected subgraph, we can split the vertices of each such subgraph arbitrarily into two sets of equal size. The endpoints of the paths represent the boxes in this game, so there are $2|L|$ boxes, and each box consists of all free edges between one endpoint of the path and half of the vertices in one Hamilton-connected subgraph. This stage will last $2L=26b\ln n$ moves.

It is straightforward to conclude that following the described strategy Maker can build a Hamilton cycle. Indeed, since each subgraph is Hamilton-connected, there exists a Hamilton path between any pair of vertices within one subgraph. These paths circularly connect to paths from $\cP$ to form a Hamilton cycle.

Now we will show that Maker can follow this strategy. We perform the analysis for each stage separately.

\noindent \textbf{Stage 1.} At the beginning of the game, $I=V$ and $X=\emptyset$. For each Hamilton-con\-nect\-ed subgraph that she builds, Maker chooses $t$ vertices $\{v_1,v_2,\dots,v_t\} \in I$ that are isolated in her graph and independent in Breaker's graph. Then, $I=I\setminus \{v_1,v_2,\dots,v_t\}$ and $X=X \cup \{v_1,v_2,\dots,v_t\}$. Since she plays on the set $X$ in every second move, this game can be treated as $(1:2b)$ Hamilton-connected subgraph game.
In her first move, Maker selects the first $t$ vertices $\{v_1,v_2,\dots,v_t\} \in I$ that are independent in Breaker's graph, which are easy to find, as there are only $b$ edges claimed on the board in total. After that, $|I|=|V|-t$ and $X=X\cup \{v_1,v_2,\dots,v_t\}$, implying $|End(\cP)|=2(n-t)$, since each vertex is treated as two endpoints of a path of length $0$. We will first take a closer look at $(2)$.\\
The proof exactly the same as the one for stage 1, part (2), in the proof of Theorem~\ref{thm:HamVerySmall} (see page~\pageref{box}), gives that when stage 1 is over, every vertex in $End(\cP)$ has Breaker's degree less than $16b\ln n$.

Now, we look at part (1). We need to prove two things: first, that Maker can build a Hamilton-connected subgraph on $t=\frac{1}{2}b\ln^2 n$ vertices, among which no edge is claimed by either of players, and second, that Maker can find such $t$ vertices that induce no edge, whenever she decides to build each of her $L$ subgraphs. In order to show that Maker can build a subgraph on $t$ vertices when playing the $(1:2b)$ game, we need to verify the conditions of Theorem~\ref{thm:FHK}. The graph Maker plays the game on is $K_t$, so the degree condition is fulfilled. Also, $\frac{t}{\ln^2 t}>\frac{2b\ln^2 n}{\ln^2 n}=2b$ for values of $b$ that we consider. This gives that Maker can build a Hamilton-connected subgraph on $V(K_t)$ in at most $1+\frac{e(K_t)}{t/\ln^2 t}\leq cb\ln^4 n$ moves, for $0<c<\frac{1}{2}$. \\
We will show that Maker can find $t$ vertices that induce no edge for each subgraph. As building each subgraph requires $cb\ln^4 n$ moves and Maker should build $L=13b\ln n$ of them, this gives in total at most $13cb^2\ln^5 n$ moves. During this number of moves, playing according to $(2)$, Maker could touch at most $2\cdot 13cb^2\ln^5 n$ vertices in $I$. Also, before selecting the $t$ vertices for her last Hamilton-connected subgraph, Maker has already removed $(L-1)\cdot t=\frac{13}{2}b^2\ln^3 n-\frac{1}{2}b\ln^2 n$ vertices from $I$. So, before choosing vertices for each subgraph, there are at least $n'=n-26cb^2\ln^5 n - \frac{13}{2}b^2\ln^3 n + \frac{1}{2}b\ln^2 n>\frac{n}{2}$ vertices in $I$ that are isolated in Maker's graph. According to Maker's stra\-te\-gy in part (2), every vertex in $End(\cP)$ has Breaker's degree less than $16b\ln n$. Applying Theorem~\ref{thm:HaSze}, we can partition $n'$ vertices into at least $16b\ln n$ independent sets, each of size at least $\frac{n'}{16b\ln n}=\Omega\left(\sqrt{n\ln^3 n}\right)>t$.

\noindent \textbf{Stage 2.} When this stage begins, $|End(\cP)|=2n-26cb^2\ln^5 n - 13b^2\ln^3 n$. Here again we look at the Box game, with Maker taking the role of BoxBreaker. The boxes in this game are vertices in $End(\cP)$, which have at least $|End(\cP)|-16b\ln n$ elements each. The difference here is that BoxMaker claims $4b$ elements of the board in each move, since Maker responds to $b$ edges of Breaker in the real game. Formally, Maker plays the game $B(|End(\cP)|,|End(\cP)|\cdot |End(\cP)|-16b\ln n, 4b, 1)$, pretending to be BoxBreaker. Here again, 
a calculation similar to the one in~(\ref{eqBox}) gives that playing on the board of order $|End(\cP)|$ until the end, BoxMaker cannot claim more than $8b\ln n$ elements in one box. This means that when stage 2 is over, there are $L$ paths in $\cP$ whose endpoints have degree in Breaker's graph less than $16b\ln n+8b\ln n=24b\ln n$.

\noindent\textbf{Stage 3.} Maker connects $L$ paths through $L$ Hamilton-connected subgraphs into a Hamilton cycle, by playing the Box game as BoxBreaker. Now there are $2L$ boxes in the game, representing each of the endpoints of $L$ paths in $\cP$. After stage~2 is over, there are less than $24b\ln n$ Breaker's edges incident to each $v \in End(\cP)$. Each box consists of all free edges between one endpoint of the path and half of the vertices in one subgraph, and so, each box is of size more than $s=\frac{t}{2}-24b\ln n=\frac{1}{4}b\ln^2 n-24b\ln n$. Each Breaker's edge is counted as claiming one element of the board, so the game played is $B(2L,2L s,b,1)$. The size of the largest box that BoxMaker could fully claim until the end of the game playing with bias $b$ is at most
\begin{displaymath}
l=\frac{b}{2L}+\frac{b}{2L-1}+\cdots+\frac{b}{1}=b\sum_{i=1}^{2L}\frac{1}{i}\le b(1+\ln 2L)<s.
\end{displaymath}
 This means that BoxMaker is unable to fully claim any box in this game before BoxBreaker claims an element in it. So, this stage ends in $2L$ moves and at its end, Maker's graph contains a Hamilton cycle.
The total number of moves in this stage is $2L=26b\ln n$, so the game lasts altogether at most
$n+\frac{13}{2}cb^2\ln^5n-\frac{13}{2}b^2\ln^3n+13b\ln n=n+O(b^2\ln^5 n)$ moves.
\end{proof}

\section{Proof of Theorem~\ref{thm:delay}}
\label{sec:breakfast}

\begin{proof}
\begin{enumerate}[$(i)$]

\item 
Breaker's strategy consists of claiming all the edges of some clique $C$ on $\frac b2$ vertices such that no vertex of $C$ is touched by Maker and maintaining this clique. Let $C_i$ be the clique of Breaker before his $i$th move. Let $u_i$ be the largest integer such that $b_i\leq b$, where $b_i:={{u_i+1}\choose {2}} + (u_i+1)|C_i|$. In his $i$th move, Breaker chooses $u_i+1$ vertices $\{v_1,v_2,\dots,v_{u_i+1}\}\in V(K_n)\backslash V(C_i)$ such that $d_M(v_j)=0$, for $1\leq j \leq u_i+1$ and claims the edges $\{\{v_j,v_k\}: 1\leq j<k\leq u_i+1\}\cup\{\{v_j,v\}:1\leq j\leq u_i+1, v\in V(C_i)\}$. He also claims $b-b_i$ arbitrary edges which we will disregard in our analysis. Maker, on the other hand, can touch at most one vertex from $C_i$ in his following move, so right before Breaker's $(i+1)$st move, $|C_{i+1}|\geq |C_{i}|+u_i$. It is easy to verify that $u_i\geq 1$ while $|C_i|< \frac b2$, and after that $|C_{i+1}|\geq |C_{i}|$, provided there is at least one vertex in $V(K_n)\backslash V(C_i)$ isolated in Maker's graph. What we need to show is that there are enough vertices for Breaker to create a clique with $\frac b2$ vertices. By the definition of $u_i$, until $|C_i|\leq \frac b 4-2$, $u_i\geq 3$. While $|C_i|\leq \frac b3-1$, $u_i\geq 2$ and if $|C_i|\leq \frac b2-1$, then $u_i\geq 1$. It takes at most
\begin{displaymath}
\frac{\frac{b}{4}-2}{3}+\frac{\frac{b}{3}-1-\frac{b}{4}+1}{2}+\frac{\frac{b}{2}-1-\frac{b}{3}}{1}=\frac {7b-40}{24}
\end{displaymath}
moves to create a clique. Note that in any of his moves, Breaker can enlarge his clique by more than stated number of vertices. Consequently, the number of rounds will decrease. In any case, the number of vertices touched by any of the players is upper bounded by
\begin{displaymath}
6\frac{\frac{b}{4}-2}{3}+5\frac{\frac{b}{3}-1-\frac{b}{4}+1}{2}+4\frac{\frac{b}{2}-1-\frac{b}{3}}{1}=\frac{33b-192}{24}<n
\end{displaymath} 
knowing that $b=o(n)$. 
When it is no longer possible for Breaker to add new  vertices, isolated in Maker's graph, to his clique, Maker needs at least one move to connect each vertex $w\in V(C)$ to some vertex $v\in V(K_n)\setminus V(C)$, which has degree at least one in Maker's graph. When this happens, $d_M(v)$ increases by one. Maker has to claim at least $\frac b2$ edges to touch all vertices in Breaker's clique. In the smallest graph that contains a perfect matching all vertices have degree one. The double number of extra edges that Maker claims is thus $\sum_{v\in V(K_n)}(d_M(v)-1)\geq \frac b2$, by the aforementioned analysis. The number of extra edges Maker has claimed in this game is thus $\frac{\sum_{v\in V(K_n)}(d_M(v)-1)}2\geq \frac b4$.

\item  The proof is similar to the one for \textit{(i)}. Let $i$ be the smallest integer such that $C_i$ is the clique in Breaker's graph of order $\frac b2$, created in the same way as in \textit{(i)} and for every $w \in V(C_i)$ it holds $d_M(w)=0$. In the Hamilton Cycle game, Breaker considers a vertex $w\in C_i$ to be removed from his clique only if $d_M(w)=2$. So, when $|C_i|=\frac b2$ and for every $w \in V(C_i)$, $d_M(w)=0$ holds, Maker needs two moves to remove a vertex from $C$. This means that in $k$th move, $k>i$, Breaker chooses one untouched vertex $v\in V(K_n)\backslash V(C_k)$ and enlarges his clique by one. So, after each two moves, at most one vertex can be removed, and $|C_{k+2}|\geq |C_k|+1$ holds. This is clearly possible while $|C|\leq b$. What remains to be proved is that there are enough untouched vertices until Breaker creates a clique of order $b$. From $\textit{(i)}$ we know that at most $\frac {33b-192}{24}$ vertices are touched until the clique of order $\frac b2$ is created in Breaker's graph. After that in at most $b$ more moves Breaker enlarges his clique to order $b$, and at that point the total of at most $\frac{57b-192}{24}<n$ vertices are touched. When $d_M(v)\geq 1$ for every $v\in V(K_n)\backslash V(C)$, Breaker still enlarges his clique by adding to it a vertex of degree 1 in Maker's graph. However, from that point on, Maker needs one move to remove such a vertex of degree one from $C_{k+1}$ and $|C_{k+1}|\geq |C_{k}|$. This implies that Breaker can maintain a clique in his graph of order $b$ until for all vertices $v\in V(K_n)\backslash V(C)$ it holds that $d_M(v)\geq2$.\\
For every vertex $w\in V(C)$, $d_M(w)\leq 1$. In order to connect a
vertex $w \in V(C)$ to some $v\in V(K_n)\backslash V(C)$ Maker needs at
least one move. In a graph that contains a Hamilton cycle all the
vertices have degree at least 2. So, the number of extra moves that
Maker has made when the game is over is $\frac{\sum_{v\in
V(K_n)}(d_M(v)-2)}2$. By the given strategy, the sum $\sum_{v\in V(K_n)}(d_M(v)-2)$ grows by one for every $w \in V(C)$, and thus the number of extra moves is
$\frac{\sum_{v\in V(K_n)}(d_M(v)-2)}2 \geq \frac b2$.
\end{enumerate}
\end{proof}

\acknowledgements
We would like to thank Asaf Ferber and Dan Hefetz for the fruitful discussions at early stages of this research project. We also thank Dennis Clemens for helpful comments. Finally, we are very grateful to the anonymous referees whose remarks greatly improved the quality of this paper.

\nocite{*}
\bibliographystyle{abbrvnat}
\bibliography{fast_dmtcs-new}
\label{sec:biblio}

\appendix

\section{Proof of Theorem~\ref{thm:degreeGame}}

\begin{proof} The proof is very similar (in fact, almost identical) to the proof of~\cite[Theorem 1.2]{GS}, so we omit some of the calculations. 
At any point of the game, for every vertex
$v\in V(G)$, let $\danger(v) := d_B(v) - 2b \cdot d_M(v)$ be the
\emph{danger value} of $v$. For a subset $X\subseteq V(G)$, define
$\avdanan(X)=\frac{\sum_{v\in X} \danger(v)}{|X|}$, the average
danger of vertices in $X$. A vertex $v\in V(G)$ is called
\emph{dangerous} if $d_M(v)\leq c-1$.

The game ends when either all the vertices have degree at least $c$
in Maker's graph (and Maker won) or there exists a dangerous vertex
$v\in V(G)$ for which $\danger(v)>b(2\ln n+1)$ (and Maker failed the degree condition) or $d_B(v)\geq
d_G(v)-c+1$ (and Breaker won). Note that since
\begin{align}
d_G(v)-c+1-2b\cdot(c-1) &\geq d_G(v)-c-2bc>b(2\ln n+1),
\end{align}
it is enough to say that Maker fails if $\danger(v)>b(2\ln n+1)$
for some vertex $v\in V(G)$ with $d_M(v)\leq c-1$.

{\bf Maker's strategy $S_M$:} Before her $i$th move Maker identifies
a dangerous vertex $v_i$ with
\begin{align*} \danger(v_i)=\max\{\danger(v): v\in V(G) \textrm{ and }
v \textrm{ is dangerous}\}, \end{align*} and claims an arbitrary free
edge $\{v_i,u_i\}$, where ties are broken arbitrarily.

Suppose towards a contradiction that Breaker has a strategy $S_B$ by which Maker, who plays according to the strategy $S_M$ as
suggested above, fails. That is, playing according $S_B$, Breaker can
ensure that at some point during the game, there exists a dangerous
vertex $v\in V(G)$ for which $\danger(v)>b(2\ln n+1)$.

Let $s$ be the length of this game and let $A=\{v_1,v_2,\dots,v_s\}$ be the set of active vertices which contains all the vertices in Maker's graph of degree less than $c$ that Maker selected as the most dangerous. Note here that vertices $v_1,v_2,\dots, v_s$ do not have to be distinct vertices from $V(K_n)$, since it takes $c$ moves to remove a vertex from the set of active vertices. So, $A$ can have less than $s$ elements. By strategy $S_M$, in her $i$th move, Maker claims an edge incident with $v_i$ (for all $i$ except for $i=s$, as the game is considered to be over before her $s$th move). For $0\le i\le s-1$, let $A_{i}=\{v_{s-i},v_{s-i+1},\dots, v_s\}$.

Following the notation of~\cite{GS}, let $\danger_{M_i}(v)$ and $\danger_{B_i}(v)$ denote danger values of vertex $v\in V(K_n)$ immediately before $i$th move of Maker, respectively Breaker.

Analogously to the proof of Theorem 1.2 by~\cite{GS}, we state the following lemmas. Next lemma is useful for estimating the change in average danger value after Maker's move.

\begin{lemma}[\cite{GS}, Lemma 3.3]
\label{lem:chM}
Let $i$ be an integer, $1\leq i \leq s-1$.
\begin{itemize}
\item[(i)] If $A_{i} \neq A_{i-1}$, then $\avdanan_{M_{s-i}}(A_{i}) - \avdanan_{B_{s-i+1}}(A_{i-1}) \geq 0$.
\item [(ii)] If $A_{i} = A_{i-1}$, then $\avdanan_{M_{s-i}}(A_{i}) - \avdanan_{B_{s-i+1}}(A_{i-1}) \geq \frac{2b}
{|A_{i}|}.$
\end{itemize}
\end{lemma}
 To estimate the change in average danger value after Breaker's move, we use the following lemma.
\begin{lemma}[\cite{GS}, Lemma 3.4(i)]
\label{lem:chB}
Let $i$ be an integer, $1\leq i \leq s-1$. Then
\begin{displaymath}
\avdanan_{M_{s-i}}(A_{i}) - \avdanan_{B_{s-i}}(A_{i}) \leq \frac{2b}
{|A_{i}|}.
\end{displaymath}
\end{lemma}

Combining Lemmas~\ref{lem:chM} and~\ref{lem:chB} we obtain the following corollary which estimates the change in average danger value after a whole round is played.

\begin{corollary}[\cite{GS}, Corollary 3.5]
\label{cor:diffD}
Let $i$ be an integer, $1\leq i \leq s-1$.
\begin{itemize}
\item[(i)] If $A_{i} = A_{i-1}$, then $\avdanan_{B_{s-i}}(A_{i}) - \avdanan_{B_{s-i+1}}(A_{i-1}) \geq 0$.
\item [(ii)] If $A_{i} \neq A_{i-1}$, then $\avdanan_{B_{s-i}}(A_{i}) - \avdanan_{B_{s-i+1}}(A_{i-1}) \geq -\frac{2b}
{|A_{i}|}$,
\end{itemize}
\end{corollary}

To complete the proof, we want to show that before Breaker's first move $\danger_{B_1}(A_{s-1})>0$, thus obtaining a contradiction.\\
Let $r$ denote the number of distinct vertices in $A_1$ and let $i_1<i_2<\dots<i_{r-1}$ be the indices for which $A_{i_j}\neq A_{i_{j}-1}$ holds, for $1\leq j\leq r-1$. Then $|A_{i_j}|=j+1$.

Recall that since Maker fails in her $s$th move,
the danger value of $v_s$ immediately before $B_s$ is
\begin{equation}
\danger_{B_s}(v_s)>2b\ln n.\label{dangBgI0}
\end{equation}

We have that
\begin{eqnarray}\label{eq:lastlinehugecalc}
\avdanan_{B_1}(A_{s-1}) & = & \avdanan_{B_s}(A_{0}) +
\sum_{i=1}^{s-1}
\left( \avdanan_{B_{s-i}}(A_{i}) - \avdanan_{B_{s-i+1}}(A_{i-1}) \right) \nonumber \\
& \geq & \avdanan_{B_s}(A_{0}) + \sum_{j=1}^{r-1} \left(
\avdanan_{B_{s-i_j}}(A_{i_j}) - \avdanan_{B_{s-i_j+1}}(A_{i_j-1})
\right)\nonumber \\
&& \quad\quad\quad \mbox{[by Corollary~\ref{cor:diffD} $(i)$]} \nonumber \\
& \geq & \avdanan_{B_s}(A_{0}) - \sum_{j=1}^{r-1} \frac{2b}{j+1} \nonumber \\
&& \quad\quad\quad  \mbox{[by Corollary~\ref{cor:diffD}$(ii)$]} \nonumber \\
& \geq & \avdanan_{B_s}(A_{0}) - 2b\ln n \nonumber\\
& > & 0.
\end{eqnarray}
\end{proof}

\section{Proof of Theorem~\ref{thm:HamNotFast} }

\begin{proof} The proof of this theorem is almost identical to the proof of~\cite[Theorem 1]{K}, so we omit most of the details. Throughout the proof we assume that the edges of $H$ were claimed by Breaker.

For given $\varepsilon$, we take $\delta=\delta(\varepsilon)\leq \varepsilon/4$, and $n_0:=n_0(\delta,\varepsilon)=e^{(10/\delta)^5}$.

Similarly to the proof of~\cite[Theorem 1]{K}, we also need to set
\begin{displaymath}
 \gamma_0=\gamma_0(n)=\frac{1}{\ln^{0.49}n} \text{ and }
k_0=k_0(n)=\gamma_0n=\frac{n}{\ln^{0.49}n}.
\end{displaymath}
Given a graph $G$, an edge $e\not \in E(G)$ is called a \emph{booster} if its addition to $G$ creates either a Hamiltonian graph, or a graph whose maximum path is longer than the one in $G$.

Maker's strategy consists of the following three main stages.

{\bf Stage 1.} In this stage, Maker creates a $k_0$\emph{-expander}, in at most $12n$ moves, that is, after this stage,
Maker's graph $M$ satisfies the following property:
\begin{displaymath}
 |N_M(X)\setminus X|\geq 2|X|,\text{ for every } X\subset V(K_n)\text{ of size }
|X|\leq k_0.
\end{displaymath}

{\bf Stage 2.} Maker turns her expander into a connected graph within at most $n$ moves.

{\bf Stage 3.} Maker turns the connected expander into a Hamiltonian graph within at most $n$ moves.

Following the proof of~\cite[Theorem 1]{K}, we can prove that Maker can complete all the stages.

{\bf Stage 1.} In order to describe the strategy of Maker in this stage, for every vertex $v\in V(K_n)$ let us define the function $dang(v):=d_B(v)-2b\cdot d_M(v)$. The strategy of Maker is the mo\-di\-fied strategy from~\cite[Theorem~1.2]{GS}, already used in the proof of~\cite[Theorem~1]{K}. Maker's goal is to achieve minimum degree 12 in his graph in $12n$ moves. The strategy \emph{S} is the following: While there exists a vertex in Maker's graph of degree less than 12, Maker chooses a vertex $v$ of degree less than $12$ with the largest value of $dang(v)$ and claims a random free edge $e$ incident to it. Similarly as it is done in~\cite[Lemma~3]{K}, the argument of~\cite[Theorem~1.2]{GS} can be used to obtain that using strategy \emph{S} for every vertex $v\in V(K_n)$ Maker can claim at least 12 edges incident to it before Breaker has claimed $(1-2\delta)n$ edges incident to $v$ (for details, we refer the reader to~\cite{GS,K}). Note now, that we consider that all the edges of $H$ were claimed by Breaker. This gives that for every vertex $v\in V(K_n)$ at the end of this stage in Breaker's graph it holds: \\
\begin{equation}
d_B(v)\leq (1-\delta)n. \label{degB}
\end{equation}

Now, we show that after this stage, the Maker's graph $M$ is indeed a $k_0$-expander. Following the argument of~\cite[Lemma~4]{K}, we will suppose that $M$ is not a $k_0$-expander. Then, there exists a subset $A$ of size $|A|=i\leq k_0$ in $M$ at the end of this stage such that $N_M(A)$ is contained in a set $B$ of size at most $2i-1$. As the minimum degree in $M$ is 12, we can take that $i\geq 5$. Also, there are at least $6i$ edges in $M$ incident to $A$. Let $e=\{p,q\}$ be an edge that Maker has chosen with $p\in A\cup B$. As $d_B(v)\leq (1-\delta)n$ holds for every vertex $v\in V(K_n)$ (by~\ref{degB}), and Maker's degree was at most 11 at the time he chose $e$, there were at least $\delta n-12$ free edges incident with $p$. When choosing $e$, the probability that $q\in A\cup B$ is at most $\frac{|A\cup B|-1}{\delta n - 12}$, independently of the previous course of the game. Consequently, the probability that all $6i$ edges belong to $A\cup B$ is at most $\left(\frac{3i-2}{\delta n -12}\right)^{6i}$. Taking the sum over all relevant values of $i$, in the same way as in~\cite[Theorem 1]{K}, we obtain that the probability that the graph $M$ is not a $k_0$-expander is at most
\begin{displaymath}
\sum_{5\leq i \leq k_0}{n \choose i}{{n-i} \choose {2i-1}}\left(\frac{3i-2}{\delta n -12}\right)^{6i}\leq \sum_{5\leq i \leq k_0}\left[4^5e^3\left(\frac{i}{n}\right)^3\frac{1}{\delta^6}\right]^i.
\end{displaymath}

Let $g(i)$ denote the $i$th term of the above sum (same as in~\cite{K}). Similar calculations as in~\cite{K} give us that for $5\leq i \leq \sqrt{n}$, $g(i)<\left(O(1)(\ln^{6/5} n)n^{-3/2}\right)^6=o(1/n)$. Also, for $\sqrt{n}<i\leq k_0$, $g(i)\leq \left(\frac{4^5e^3\gamma_0^3}{\delta^6}\right)^{\sqrt{n}}=o(1/n)$. Thus, almost surely, Maker's graph satisfies the required property after stage~1. Notice that after stage~1, we consider that Breaker has claimed at most $12nb+e(H)<\frac{13n^2}{\ln n}$ edges.

\textbf{Stage 2.} Using~\cite[Lemma 2]{K}, every connected component of Maker's graph is of size at least $c=\frac{3n}{\ln^{0.49}n}$, thus there are $\Omega\left(\frac{n^2}{\ln^{0.98}n}\right)$ edges between any two connected components in the complete graph. Maker needs at most $n/c-1$ moves to connect all the components into one, during which time Breaker can additionally claim at most $bn/c$ edges. In total, after these two stages, Breaker claimed less than $\frac{14n^2}{\ln n}$ edges, thus we conclude that most of the edges between any two Maker's components are still free and Maker can merge all the components into one component and complete stage~2 in her next $n/c<n$ moves.

\textbf{Stage 3.} When stage~2 ends, Maker's graph $M$ is connected and either it is already Hamiltonian or, by~\cite[Lemma 1]{K}, it has at least $k_0^2/2=\frac{n^2}{2\ln^{0.98}n}$ boosters (as the expansion property does not vanish when adding additional edges). From the definition of a booster, we conclude that Maker needs to add at most $n$ boosters to make her graph $M$ Hamiltonian. After three stages, which last for at most $14n$ moves, Breaker could have claimed less than $\frac{15n^2}{\ln n}<\frac{k_0}{2}$ edges in total (including the edges of $H$). So, in each of the following at most $n$ moves, Maker can claim a booster and make her graph Hamiltonian.
\end{proof}

\end{document}